\documentclass{amsart}
\usepackage{amsfonts}
\usepackage{graphicx}
\usepackage{amsmath}
\usepackage{amssymb}
\usepackage{amsthm}
\usepackage{enumerate}
\usepackage[all, cmtip]{xy} 
\newtheorem{theorem}{Theorem}

\newtheorem{corollary}[theorem]{Corollary}

\newtheorem{lemma}[theorem]{Lemma}

\newtheorem{remark}[theorem]{Remark}

\newcommand{\Ker}{\operatorname{Ker}}
\newcommand{\End}{\operatorname{End}}
\newcommand{\Soc}{\operatorname{Soc}}
\newcommand{\SI}{\operatorname{SI}}
\newcommand{\CS}{\operatorname{CS}}
\begin{document}
\title{Rings of Invariant module Type and Automorphism-invariant Modules}
\author{Surjeet Singh}
\address{House No. 424, Sector No. 35 A, Chandigarh-160036, India}
\email{ossinghpal@yahoo.co.in}
\author{Ashish K. Srivastava}
\address{Department of Mathematics and Computer Science, St. Louis University, St.
Louis, MO-63103, USA}
\email{asrivas3@slu.edu}
\dedicatory{Dedicated to T. Y. Lam on his $70$ th Birthday}

\begin{abstract}
A module is called automorphism-invariant if it is invariant under any automorphism of its injective hull. In [Algebras for which every indecomposable right module is invariant in its injective envelope, Pacific J. Math., vol. 31, no. 3 (1969), 655-658] Dickson and Fuller had shown that if $R$ is a finite-dimensional algebra over a field $\mathbb F$ with more than two elements then an indecomposable automorphism-invariant right $R$-module must be quasi-injective. In this paper we show that this result fails to hold if $\mathbb F$ is a field with two elements. Dickson and Fuller had further shown that if $R$ is a finite-dimensional algebra over a field $\mathbb F$ with more than two elements, then $R$ is of right invariant module type if and only if every indecomposable right $R$-module is automorphism-invariant. We extend the result of Dickson and Fuller to any right artinian ring. A ring $R$ is said to be of right automorphism-invariant type (in short, RAI-type) if every finitely generated indecomposable right $R$-module is automorphism-invariant. In this paper we completely characterize an indecomposable right artinian ring of RAI-type.
\end{abstract}

\maketitle

\bigskip

\bigskip

\bigskip

\section{Introduction}

\noindent All our rings have identity element and modules are right unital. A right $R$-module $M$ is called an {\it automorphism-invariant module} if $M$ is invariant under any automorphism of its injective hull, i.e. for any automorphism $\sigma$ of $E(M)$, $\sigma(M)\subseteq M$ where $E(M)$ denotes the injective hull of $M$. 

Indecomposable modules $M$ with the property that $M$ is invariant under any automorphism of its injective hull were first studied by Dickson and Fuller in \cite{DF} for the particular case of finite-dimensional algebras over fields $\mathbb F$ with more than two elements. But for modules over arbitrary rings, study of such a property has been initiated recently by Lee and Zhou in \cite{LZ}. The dual notion of these modules has been proposed by Singh and Srivastava in \cite{SS}. 

The obvious examples of the class of automorphism-invariant modules are quasi-injective modules and pseudo-injective modules. Recall that a module $M$ is said to be {\it $N$-injective} if for every submodule $N_1$ of the module $N$, all homomorphisms $N_1\rightarrow M$ can be extended to homomorphisms $N\rightarrow M$. A right $R$-module $M$ is {\it injective} if $M$ is $N$-injective for every $N\in $ Mod-$R$. A module $M$ is said to be {\it quasi-injective} if $M$ is $M$-injective. A module $M$ is called pseudo-injective
if every monomorphism from a submodule of $M$ to $M$ extends to an endomorphism of $M$.

\bigskip

\noindent Thus we have the following hierarchy;

\bigskip

\noindent injective $\implies$ quasi-injective $\implies$ pseudo-injective $\implies$ automorphism-invariant 

\bigskip

\noindent It is well known that a quasi-injective module need not be injective. In \cite{Teply} Teply gave construction of a pseudo-injective module which is not quasi-injective. We do not know yet an example of an automorphism-invariant module which is not pseudo-injective. 

Dickson and Fuller \cite{DF} studied automorphism-invariant modules in case of finite-dimensional algebras over a field $\mathbb F$ with more than two elements. They proved that if $R$ is a finite-dimensional algebra over a field $\mathbb F$ with more than two elements then an indecomposable automorphism-invariant right $R$-module must be quasi-injective. We show that this result fails to hold if $\mathbb F$ is a field with two elements. A ring $R$ is said to be of {\it right invariant module type} if every indecomposable right $R$-module is quasi-injective. Dickson and Fuller had further shown that if $R$ is a finite-dimensional algebra over a field $\mathbb F$ with more than two elements, then $R$ is of right invariant module type if and only if every indecomposable right $R$-module is automorphism-invariant. We extend the result of Dickson and Fuller to any right artinian ring. 

We call a ring $R$ to be of {\it right automorphism-invariant type} (in short, RAI-type), if every finitely generated indecomposable right $R$-module is automorphism-invariant. In this paper we study the structure of indecomposable right artinian rings of RAI-type.   

Lee and Zhou in \cite{LZ} asked whether every automorphism-invariant module is pseudo-injective. In this paper we show that the answer is in the affirmative for modules with finite Goldie dimension. 

We also prove that a simple right noetherian ring $R$ is a right $\SI$ ring if and only if every cyclic singular right $R$-module is automorphism-invariant.

Before presenting the proofs of these results, let us recall some basic definitions and facts. A module $M$ is said to have finite Goldie (or uniform) dimension if it does not contain an infinite direct sum $\displaystyle\bigoplus_{n\in\mathbb N}M_n$ of non-zero submodules. 

A module $M$ is called {\it directly-finite} if $M$ is not isomorphic to a proper summand of itself. Clearly, a module with finite Goldie dimension is directly-finite. A module $M$ is called a {\it square} if $M\cong X\oplus X$ for some module $X$; and a module is called {\it square-free} if it does not contain a non-zero square. 

A module $M$ is said to have the {\it internal cancellation property} if whenever $M=A_1\oplus B_1=A_2\oplus B_2$ with $A_1\cong A_2$, then $B_1\cong B_2$. For details on internal cancellation property, the reader is referred to \cite{l3}. Now, if an injective module $M$ is directly-finite, then it has internal cancellation property (see \cite[Theorem 1.29]{MM}).

A module $M$ is said to be {\it uniserial} if any two submodules of $M$ are comparable with respect to inclusion. A ring $R$ is called a {\it right uniserial ring} if $R_R$ is a uniserial module. Any direct sum of uniserial modules is called a {\it serial} module. A ring $R$ is said to be a {\it right serial ring} if the module $R_R$ is serial. A ring $R$ is called a {\it serial ring} if $R$ is both left as well as right serial.

\noindent If $A$ is an essential submodule of $B$, then we denote it as $A\subseteq_e B$. For any module $M$, we define $Z(M)=\{x \in M: ann_r(x)\subseteq_{e} R_R\}$. It can be easily checked that $Z(M)$ is a submodule of $M$. It is called the {\it singular submodule} of $M$. If $Z(M)=M$, then $M$ is called a {\it singular module}. If $Z(M)=0$, then $M$ is called a {\it non-singular module}.

\bigskip

\noindent Consider the following three conditions on a module $M$;\\
C1: Every submodule of $M$ is essential in a direct summand of $M$.\\
C2: Every submodule of $M$ isomorphic to a direct summand of $M$ is itself a direct summand of $M$.\\
C3: If $N_1$ and $N_2$ are direct summands of $M$ with $N_1\cap N_2=0$ then $N_1\oplus N_2$ is also a direct summand of $M$.

\bigskip

\noindent A module $M$ is called a {\it continuous module} if it satisfies conditions C1 and C2. A module $M$ is called {\it $\pi$-injective} (or {\it quasi-continuous}) if it satisfies conditions C1 and C3. A module $M$ is called a {\it $\CS$ module} (or {\it extending module}) if it satisfies condition C1.

\noindent In general, we have the following implications\label{hierarchy}. 
\begin{center}
Injective $\Longrightarrow$ Quasi-injective $\Longrightarrow$ Continuous
$\Longrightarrow$ $\pi$-injective $\Longrightarrow$ $\CS$\\ 
\end{center}

\bigskip

\noindent The socle of a module $M$ is denoted by $\Soc(M)$. A right $R$-module $M$ is called {\it semi-artinian} if for every submodule $N\neq M$, $\Soc(M/N)\neq 0$. A ring $R$ is called right semi-artinian if $R_R$ is semi-artinian. We denote by $J(R)$, the Jacobson radical of a ring $R$. For any term not defined here, the reader is referred to \cite{JST}, \cite{l1}, \cite{l2}, and  \cite{MM}.

\bigskip

\section{Basic Facts about Automorphism-invariant Modules}

\bigskip
 
\noindent Lee and Zhou proved the following basic facts about automorphism-invariant modules \cite{LZ}. 
\begin{itemize}
\item A module $M$ is automorphism-invariant if and only if every isomorphism between any two essential submodules of $M$ extends to an automorphism of $M$.
\item A direct summand of an automorphism-invariant module is automorphism-invariant.
\item If for two modules $M_1$ and $M_2$, $M_1\oplus M_2$ is automorphism-invariant, then $M_1$ is $M_2$-injective and $M_2$ is $M_1$-injective.
\item Every automorphism-invariant module satisfies the property C3.
\item A CS automorphism-invariant module is quasi-injective. 
\end{itemize}      

\bigskip

\bigskip

\section{Results}

\noindent Dickson and Fuller in \cite{DF} considered a finite-dimensional algebra $R$ over a field $\mathbb F$ with more than two elements and proved that if an indecomposable right $R$-module $M$ is automorphism-invariant, then $M$ is quasi-injective. They further obtained the following. 

\begin{theorem} $($Dickson and Fuller, \cite{DF}$)$ \label{dickson}
Let $R$ be a finite-dimensional algebra over a field $\mathbb F$ with more than two elements. Then the following statements are equivalent;

\bigskip

\begin{enumerate}[(i)]
\item Each indecomposable right $R$-module is automorphism-invariant.

\bigskip

\item Each indecomposable right $R$-module is quasi-injective.

\bigskip

\item Each indecomposable right $R$-module has a square-free socle.  
\end{enumerate}
\end{theorem}

\bigskip

\noindent  We will provide an example to show that if $R$ is a finite-dimensional algebra over a field $\mathbb F$ with two elements, then an indecomposable automorphism-invariant right $R$-module need not be quasi-injective. 

First, note that in an artinian serial ring $R$, any indecomposable summand of $R_R$ of maximum length is injective. Thus if $T_n(D)$ is the upper triangular matrix ring over a division ring $D$, then $e_{11}T_n(D)$ is injective and uniserial. 

\bigskip

\noindent {\bf Example.} Let $R=\left[ 
\begin{array}{ccc}
\mathbb F & \mathbb F & \mathbb F \\ 
0 & \mathbb F & 0 \\ 
0 & 0 & \mathbb F  \\ 
\end{array}
\right] $ where $\mathbb F$ is a field of order 2. 

\noindent We know that $R$ is a left serial ring. Note that $e_{11}R$ is a local module, $e_{12}\mathbb F\cong e_{22}R$, $e_{13}\mathbb F\cong e_{33}R$ and $e_{11}J(R)=e_{12}\mathbb F\oplus e_{13}\mathbb F$, a direct sum of two minimal right ideals. So the injective hull of $e_{11}R$ is $E(e_{11}R)=E_1\oplus E_2$, where $E_1=E(e_{12}\mathbb F)$ and $E_2=E(e_{13}\mathbb F)$. 

Now set $A=ann_r(e_{12}\mathbb F)$. Then $A=e_{13}\mathbb F+e_{33}\mathbb F$. Thus $\overline{R}=R/A\cong \left[ 
\begin{array}{cc}
\mathbb F & \mathbb F\\ 
0 & \mathbb F\\ 
\end{array}
\right] =S$. Denote the first row of $S$ by $S_1$. It may be checked that $S_1$ is injective. As $\mathbb F$ has only two elements, $S_1$ has only two endomorphisms, zero and the identity. Take the pre-image $L_1$ of $S_1$ in $\overline{R}$. It is uniserial with composition length 2, and $e_{12}\mathbb F$ naturally embeds in $L_1$. There is no mapping of $e_{13}\mathbb F$ into $L_1$. It follows that $L_1$ is $e_{11}R$-injective and $e_{12}\mathbb F$-injective. As $e_{22}R\cong e_{12}\mathbb F$, $L_1$ is $e_{22}R$-injective. There is no map from $e_{33}R$ into $L_1$ so it is also $e_{33}R$-injective. Hence $L_1$ is injective. Thus $E_1= L_1$ and its ring of endomorphisms has only two elements. 

If $B = ann_r(e_{13}\mathbb F)$, then $B = e_{12}\mathbb F + e_{22}\mathbb F$. Thus $R/B\cong \left[ 
\begin{array}{cc}
\mathbb F & \mathbb F \\ 
0 & \mathbb F\\ 
\end{array}
\right]$. The pre-image of $S_1$ in $R/B$ is $L_2$, which is uniserial, and injective. We have $E_2 \cong L_2$ and its ring of endomorphism has only two elements. 

Note that $e_{11}R$ has all its composition factors non-isomorphic, both $L_1$ and $L_2$ have composition length 2 with $\frac{L_1}{L_1J(R)}\cong \frac{e_{11}R}{e_{11}J(R)}$, $L_1J(R)\cong e_{22}R$, $\frac{L_2}{L_2J(R)}\cong \frac{e_{11}R}{e_{11}J(R)}$, and $L_2J(R)\cong e_{33}R$. Thus $L_1, L_2$ have isomorphic tops but non-isomorphic socles. 

Suppose there exists a non-zero mapping $\sigma: L_1\rightarrow L_2$. Then $\sigma(L_1)=L_2J(R)$. Thus $\frac{e_{11}R}{e_{11}J(R)}\cong e_{33}R$, which is a contradiction. Therefore, there is no non-zero map between $L_1$ and $L_2$. 

Hence the only automorphism of $L_1\oplus L_2$ is the identity. So $e_{11}R$ is trivially automorphism-invariant but it is not uniform. Then clearly $e_{11}R$ is not quasi-injective as an indecomposable quasi-injective module must be uniform. 

Thus, this ring $R$ is an example of a finite-dimensional algebra over a field $\mathbb F$ with two elements such that there exists an indecomposable right $R$-module which is automorphism-invariant but not quasi-injective. ~\hfill$\square$    

\bigskip

\noindent Next, we proceed to extend the result of Dickson and Fuller \cite{DF} to any right artinian ring. But, first we obtain a useful result on decomposition property of automorphism-invariant modules. 

We will show that under certain conditions a decomposition of injective hull $E(M)$ of an automorphism-invariant module $M$ induces a natural decomposition of $M$.

We will denote the identity automorphism on any module $M$ by $I_M$.

\begin{lemma} \label{s1}
Let $M$ be an automorphism-invariant right module over any ring $R$. If $E(M)=E_1\oplus E_2$ and $\pi_1:E(M)\rightarrow E_1$ is an associated projection, then $M_1=\pi_1(M)$ is also automorphism-invariant.
\end{lemma}

\begin{proof}
Let $E(M)=E_1\oplus E_2$ and $M_1=\pi_1(M)$, where $\pi_1:E(M)\rightarrow E_1$ is a projection with $E_2$ as its kernel. Let $\sigma_1$ be an automorphism of $E_1$ and $x_1\in M_1$. For some $x\in M$, and $x_2\in E_2$, we have $x=x_1+x_2$. Now $\sigma=\sigma_1 \oplus I_{E_2}$ is an automorphism of $E$. Thus $\sigma(x)=\sigma_1(x_1)+x_2\in M$, which gives $\sigma_1(x_1)\in M_1$. Hence $M_1$ is automorphism-invariant. 
\end{proof}

\begin{lemma} \label{new-imp}
Let $M$ be an automorphism-invariant right module over any ring $R$. Let $E(M)=E_1\oplus E_2$ such that there exists an automorphism $\sigma_1$ of $E_1$ such that $I_{E_1}-\sigma_1$ is also an automorphism of $E_1$. Then \[M=(M\cap E_1) \oplus (M\cap E_2).\]
\end{lemma}

\begin{proof}
Set $E=E(M)$. Set $I_E = I_{E_1}\oplus I_{E_2}$, and $\sigma = \sigma_1 \oplus I_{E_2}$. Clearly, both $I_E$ and $\sigma$ are automorphisms of $E$. Since $M$ is assumed to be an automorphism-invariant module, $M$ is invariant under automorphisms $I_E$ and $\sigma$. Consequently, $M$ is invariant under $I_E-\sigma$ too. Note that $I_E-\sigma=(I_{E_1}-\sigma_1)\oplus 0$. Thus $(I_E-\sigma)(M)=(I_{E_1}-\sigma_1)(M)\subseteq M$. Let $\pi_1: E\rightarrow E_1$ and $\pi_2: E\rightarrow E_2$ be the canonical projections. Set $M_1=\pi_1(M)$ and $M_2=\pi_2(M)$. Now $M\cap E_1\subseteq M_1$ and $M\cap E_2\subseteq M_2$. 

Let $0\neq u_1 \in E_1$. For some $r\in R$, $0\neq u_1 r\in M$ and thus $u_1r\in M_1$. Thus $M_1\subseteq_{e} E_1$. By Lemma \ref{s1}, $M_1$ is automorphism-invariant. Therefore, $M_1=(I_{E_1}-\sigma_1)^{-1}(M_1)$. Let $x_1\in M_1$. Then, we have for some $x \in M$, $x = x_1+x_2$, $x_2 \in E_2$. Now, as $I_{E_1}-\sigma_1$ is an automorphism on $E_1$, there exists an element $y_1 \in E_1$ such that $(I_{E_1}-\sigma_1)(y_1) = x_1$, which gives $y_1 \in (I_{E_1}-\sigma_1)^{-1}(M_1) = M_1$. This yields an element $y \in M$ such that $y = y_1 + y_2$ for some $y_2\in E_2$. We get $(I_E-\sigma)(y) = (I_{E_1}-\sigma_1)(y_1) = x_1$. Thus $x_1 \in (I_{E}-\sigma)(M)$. As $(I_{E}-\sigma)(M)\subseteq M$, we get $x_1\in M$. Hence $M_1\subseteq M$. 

Now, let $u_2 \in M_2$ be an arbitrary element. For some $u_1\in M_1$, we have $u=u_1+u_2\in M$. But we have shown in the previous paragraph that $M_1\subseteq M$, so $u_1\in M$. Therefore $u_2=u-u_1\in M$. Hence $M_2 \subseteq M$. This gives $M_1\oplus M_2 \subseteq M$ and hence $M = M_1\oplus M_2$. Thus $M=(M\cap E_1) \oplus (M\cap E_2)$.
\end{proof}

\bigskip

A quasi-injective module is obviously automorphism-invariant. In the next result we give a condition under which an automorphism-invariant module must be quasi-injective.  

\begin{theorem} \label{imp2}
Let $M$ be a right module over any ring $R$ such that every summand $E_1$ of $E(M)$ admits an automorphism $\sigma_1$ such that $I_{E_1}-\sigma_1$ is also an automorphism of $E_1$, then $M$ is automorphism-invariant if and only if $M$ is quasi-injective.  
\end{theorem}

\begin{proof}  
Let $M$ be automorphism-invariant. Set $E=E(M)$. Suppose every summand $E_1$ of $E$ admits an automorphism $\sigma_1$ such that $I_{E_1}-\sigma_1$ is also an automorphism of $E_1$. 

Let $\sigma \in \End(E)$ be an arbitrary element. Since $\End(E)$ is a clean ring \cite{CKLNZ}, $\sigma=\alpha+\beta$ where $\alpha$ is an idempotent and $\beta$ is an automorphism. 

Let $E_1=\alpha E$, and $E_2=(1-\alpha)E$. Then $E=E_1\oplus E_2$. By Lemma \ref{new-imp}, we have $M=M_1\oplus M_2$ where $M_1=M\cap E_1, M_2=M\cap E_2$. 

Then clearly $\alpha(M)\subseteq M$. Since $M$ is automorphism-invariant, $\beta(M)\subseteq M$. Thus $\sigma(M)\subseteq M$. Hence $M$ is quasi-injective.  

The converse is obvious.
\end{proof}

As a consequence of this theorem, we may now deduce the following which extends the result of Dickson and Fuller \cite{DF} to any algebra (not necessarily finite-dimensional) over a 
field $\mathbb F$ with more than two elements.

\begin{corollary} \label{imp3}
Let $R$ be any algebra over a field $\mathbb F$ with more than two elements. Then the following are equivalent;
\begin{enumerate}[(i)]
\item Each indecomposable right $R$-module is automorphism-invariant. 

\bigskip

\item Each indecomposable right $R$-module is quasi-injective, that is, $R$ is of right invariant module type. 
\end{enumerate}  
\end{corollary}

\begin{proof}
Clearly, for any right $R$-module $E$, the multiplication by an element $u \in \mathbb F$ where $u\neq 0$ and $u\neq 1$ gives an automorphism $\sigma$ of $E$ such that $I_E-\sigma$ is also an automorphism of $E$. Hence the result follows from the above theorem.  
\end{proof}

\begin{corollary} (\cite{LZ}) 
Let $R$ be a ring in which $2$ is invertible. Then any automorphism-invariant module over $R$ is quasi-injective. 
\end{corollary}

\begin{proof}
Let $M$ be an automorphism-invariant right $R$-module. Let $E = E(M)$. Let $E_1$ be any summand of $E$. We have automorphism $\sigma_1 : E_1\rightarrow E_1$, given by $\sigma_1(x) = 2x$, $x \in E_1$. Clearly, $I_{E_1}-\sigma_1=-I_{E_1}$ is also an automorphism of $E_1$. By Theorem \ref{imp2}, $M$ is quasi-injective.
\end{proof}

In the next lemma we give another useful result on decomposition of automorphism-invariant modules. 

\begin{lemma} \label{tuk}
Let $M$ be an automorphism-invariant right module over any ring $R$. If $E(M)=E_1\oplus E_2\oplus E_3$, where $E_1\cong E_2$, then \[M=(M\cap E_1) \oplus (M\cap E_2) \oplus (M\cap E_3).\] 
\end{lemma}

\begin{proof}
Set $E(M)=E$. Let $E=E_1\oplus E_2\oplus E_3$. Let $\sigma:E_1\rightarrow E_2$ be an isomorphism and let $\pi_1: E\rightarrow E_1$, $\pi_2: E\rightarrow E_2$, and $\pi_3: E\rightarrow E_3$ be the canonical projections. Then $M\cap E_1\subseteq \pi_1(M)$, $M\cap E_2\subseteq \pi_2(M)$ and $M\cap E_3\subseteq \pi_3(M)$.

Let $\eta=\sigma^{-1}$. Consider the map $\lambda_1: E\rightarrow E$ given by $\lambda_1(x_1, x_2, x_3)=(x_1, \sigma(x_1)+x_2, x_3)$. Clearly, $\lambda_1$ is an automorphism of $E$. Since $M$ is automorphism-invariant, $M$ is invariant under $\lambda_1$ and $I_E$. Consequently, $M$ is invariant under $\lambda_1-I_E$. Thus $(\lambda_1-I_E)(M)\subseteq M$. Next, we consider the map $\lambda_2: E\rightarrow E$ given by $\lambda_2(x_1, x_2, x_3)=(x_1+\eta(x_2), x_2, x_3)$. This map $\lambda_2$ is also an automorphism of $E$. Thus, as explained above, $M$ is invariant under $\lambda_2-I_E$ too, that is $(\lambda_2-I_E)(M)\subseteq M$. 

Let $x=(x_1, x_2, x_3) \in M$. Then $(\lambda_1-I_E)(x)=(0, \sigma(x_1), 0) \in M$. Similarly, we have $(\lambda_2-I_E)(x)=(\eta(x_2), 0, 0) \in M$. This gives $(\lambda_1-I_E)(\eta(x_2), 0, 0)=(0, \sigma\eta(x_2), 0)=(0, x_2, 0) \in M$. Thus $\pi_2(M)\subseteq M$. Similarly, $(\lambda_2-I_E)(0, \sigma(x_1), 0)=(\eta\sigma(x_1), 0, 0)= (x_1, 0, 0)\in M$. Thus $\pi_1(M)\subseteq M$. This yields that $(0, 0, x_3)\in M$, that is, $\pi_3(M)\subseteq M$. This shows that $\pi_1(M)\oplus \pi_2(M)\oplus \pi_3(M)\subseteq M$ and therefore, $M=\pi_1(M)\oplus \pi_2(M)\oplus \pi_3(M)$. Hence $M=(M\cap E_1) \oplus (M\cap E_2) \oplus (M\cap E_3)$.  
\end{proof}

\bigskip

As a consequence of the above decomposition, we have the following for socle of an indecomposable automorphism-invariant module.

\begin{corollary} \label{soc}
If $M$ is an indecomposable automorphism-invariant right module over any ring $R$, then $\Soc(M)$ is square-free.
\end{corollary}

\begin{proof}
Let $M$ be an indecomposable automorphism-invariant module. Suppose $M$ has two isomorphic simple submodules $S_1$ and $S_2$. Then $E(M)=E_1\oplus E_2\oplus E_3$, where $E_1=E(S_1), E_2=E(S_2)$ and $E_1\cong E_2$. By Lemma \ref{tuk}, $M$ decomposes as $M=(M\cap E_1) \oplus (M\cap E_2) \oplus (M\cap E_3)$, a contradiction to our assumption that $M$ is indecomposable. Hence $\Soc(M)$ is square-free.  
\end{proof}

\bigskip

Next, we have the following for any indecomposable semi-artinian automorphism-invariant module.

\begin{corollary} \label{tuk2}
Let $R$ be any ring and let $M$ be any indecomposable semi-artinian automorphism-invariant right $R$-module. Then one of the following statements holds:
\begin{enumerate}[(i)]
\item $M$ is uniform and quasi-injective.
\item Any simple submodule $S$ of $M$ has identity as its only automorphism.
\end{enumerate}
\end{corollary}

\begin{proof}
Let $M$ be an indecomposable semi-artinian automorphism-invariant right $R$-module. Since $M$ is semi-artinian, $\Soc(M)\neq 0$. By Corollary \ref{soc}, we know that $\Soc(M)$ is square-free. Suppose $S$ is a simple submodule of $M$. Now $D = \End(S)$ is a division ring. 

Suppose $|D|>2$. Then there exists a $\sigma \in D$ such that $\sigma \neq 0$ and $\sigma \neq I_S$. Then $I_S-\sigma$ is an automorphism of $S$. Let $E=E(M)$ and $E_1=E(S)\subseteq E$. Then $E=E_1\oplus E_2$ for some submodule $E_2$ of $E$. Let $\sigma_1 \in \End(E_1)$ be an extension of $\sigma$. Then $\sigma_1$ is an automorphism of $E_1$ and $(I_{E_1}-\sigma_1)(S)=(I_S-\sigma)(S)\neq 0$. Hence $I_{E_1}-\sigma_1$ is an automorphism of $E_1$. Thus, by Lemma \ref{new-imp}, $M=(M\cap E_1)\oplus (M\cap E_2)$. As $M$ is indecomposable, we must have $M=M\cap E_1$. Therefore, $M$ is uniform. Then $\End(E(M))$ is a local ring. Therefore for any $\alpha \in \End(E(M))$, $\alpha$ is an automorphism or $I-\alpha$ is an automorphism. In any case $\alpha(M)\subseteq M$. Therefore $M$ is quasi-injective. 

Now, if $M$ is not uniform then $|D|=2$, that is $D=\End(S)\cong \frac{\mathbb Z}{2\mathbb Z}$. In this case, the only automorphism of $S$ is the identity automorphism.   
\end{proof}

\begin{remark}
Recall that an algebra $A$ is said to be of {\it finite module type} if $A$ has only a finite number of non-isomorphic indecomposable right modules. In regard to Corollary \ref{soc}, we would like to mention here that Curtis and Jans proved that if $A$ is an algebra over an algebraically closed field $\mathbb F$ such that for each indecomposable right $A$-module $M$, $\Soc(M)$ is square-free, then $A$ is of finite module type (see \cite{CJ}). This was extended by Dickson and Fuller who proved that if $A$ is an algebra over any field $\mathbb F$ such that $A$ is of right invariant module type then $A$ has finite module type \cite{DF}.
\end{remark}

\bigskip

We call a ring $R$ to be of {\it right automorphism-invariant type} (in short, RAI-type), if every finitely generated indecomposable right $R$-module is automorphism-invariant. We would like to understand the structure of right artinian rings of RAI-type.   

\begin{lemma} \label{l1}
Let $R$ be a right artinian ring of RAI-type. Let $e\in R$ be an indecomposable idempotent such that $eR$ is not uniform. Let $A$ be a right ideal contained in $\Soc(eR)$. Then $\Soc(eR)=A\oplus A'$ where $A'$ has no simple submodule isomorphic to a simple submodule of $A$ and $eR/A'$ is quasi-projective.
\end{lemma}

\begin{proof}
As $\Soc(eR)$ is square-free, $\Soc(eR)=A\oplus A'$ where $A'$ has no simple submodule isomorphic to a simple submodule of $A$. If for some $ere\in eRe$, $ereA' \not\subseteq A'$, then for some minimal right ideal $S\subset A'$, $ereS\not\subseteq A'$. This gives that $S$ is isomorphic to a simple submodule contained in $A$, a contradiction. Hence $eR/A'$ is quasi-projective.     
\end{proof}

\begin{lemma} \label{l2}
Let $R$ be a right artinian ring of RAI-type. Then any uniserial right $R$-module is quasi-projective.
\end{lemma}

\begin{proof}
Let $A$ be a uniserial right $R$-module with composition length $l(A)=n\ge 2$. We will prove the result by induction. Suppose first that $n=2$. In this case, we can take $J(R)^2=0$. For some indecomposable idempotent $e\in R$, we have $A\cong eR/B$ for some $B\subseteq \Soc(eR)$. By Lemma \ref{l1}, $A$ is quasi-projective. 

Now consider $n>2$ and assume that the result holds for $n-1$. Let $0\neq \sigma: A\rightarrow A/C$ be a homomorphism where $C\neq 0$. Suppose $\sigma$ cannot be lifted to a homomorphism  $\eta: A\rightarrow A$. Let $F=\Soc(A)$. Then $F\subseteq \Ker(\sigma)$. We get a mapping $\bar{\sigma}: \frac{A}{F}\rightarrow \frac{A}{C}$. By the induction hypothesis, there exists a homomorphism $\bar{\eta}:\frac{A}{F}\rightarrow \frac{A}{F}$ such that $\bar{\sigma}=\pi \bar{\eta}$, where $\pi: \frac{A}{F}\rightarrow \frac{A}{C}$ is a natural homomorphism. 

Let $M=A\times A$, and $N=\{(a, b) \in M: \overline{\eta}(a+F)=b+F\}$. Then $N$ is a submodule of $M$. Now there exist elements $x\in A$ and indecomposable idempotent $e\in R$ such that $A = xR$ and $xe = x$. Fix an element $y\in A$ such that $\bar{\eta}(x + F) = y + F$ and $ye = y$. Set $z = (x, y)$. Then $z \in N$ and $N_1 = zR$ is local. Let $\pi_1, \pi_2$ be the associated projections of $M$ onto the first and second components of $M$, respectively. Then $\pi_1(N_1) = A$. 

Now, we claim that $N_1$ is uniserial. If $N_1$ is not uniform, then $\Soc(N_1)=\Soc(M)$. Therefore $\Soc(N_1)$ is not square-free, which is a contradiction by Lemma \ref{soc}. Thus $N_1$ is uniform. It follows that $N_1$ embeds in $A$ under $\pi_1$ or $\pi_2$. Hence $N_1$ is uniserial. As $\pi_1(N_1) = A$, and $l(N_1)\le l(A)$, it follows that $\pi_1|_{N_1}$ is an isomorphism. Thus given any $x \in A$, there exists a unique $y \in A$ such that $(x, y) \in N_1$. We get a homomorphism $\lambda: A \rightarrow A$ such that $\lambda(x) = y$ if and only
if $(x, y) \in N_1$. Clearly $\lambda$ lifts $\overline{\eta}$ and hence it also lifts $\sigma$. This proves that $A$ is quasi-projective. 
\end{proof}

\begin{lemma} \label{l3}
Let $R$ be a right artinian ring of RAI-type. Let $A_R$ be any uniserial module. Then the rings of endomorphisms of different composition factors of $A$ are isomorphic.
\end{lemma}

\begin{proof}
Let $A$ be a uniserial right $R$-module with $l(A)=2$. Let $C=ann_r(A)$ and $\overline{R}=R/C$. As $A_R$ is quasi-projective, $A$ is a projective $\overline{R}$-module. Thus there exists an indecomposable idempotent $e\in R$ such that $A\cong \overline{e} \overline{R}$. As $\overline{R}$ embeds in a finite direct sum of copies of $A$, there exists an indecomposable idempotent $f\in R$ such that $\Soc(A)\cong \frac{\overline{fR}}{\overline{fJ(R)}}$, $\overline{e}\overline{J(R)}=\overline{exfR}$ for some $x\in J(R)$. We get an embedding $\sigma: \frac{eRe}{eJ(R)e}\rightarrow \frac{fRf}{fJ(R)f}$ defined as $\sigma(ere+eJ(R)e)=fr'f+fJ(R)f$ whenever $\overline{ere}\overline{exf}=\overline{exf}\overline{fr'f}$; $ere\in eRe, fr'f \in fRf$. Let $z=fvf\in fRf$. We get an $\overline{R}$-homomorphism $\eta: \overline{e} \overline{J(R)}\rightarrow \overline{e} \overline{J(R)}$ such that $\eta(\overline{exf})=\overline{exf} \overline{fvf}$. As $\overline{e}\overline{R}$ is quasi-injective, there exists an $\overline{R}$-homomorphism $\lambda: \overline{e} \overline{R}\rightarrow \overline{e} \overline{R}$ extending $\eta$. Now $\lambda(\overline{e})=\overline{ere}$ for some $r\in R$. Then $\overline{ere}\overline{exf}=\lambda(\overline{exf})=\eta(\overline{exf})= \overline{exf}\overline{fvf}$, which gives that $\sigma$ is onto. Hence $\frac{eRe}{eJ(R)e}\cong \frac{fRf}{fJ(R)f}$. Thus the result holds whenever $l(A) = 2$. If $l(A) = n > 2$, the result follows by induction on $n$.           
\end{proof}

\begin{lemma} \label{l4}
Let $R$ be a right artinian ring of RAI-type. Then we have the following. 
\begin{enumerate}[(i)]
\item Let $D$ be a division ring and $x\in R$. Let $xR$ be a local module such that for any simple submodule $S$ of $\Soc(xR)$, $D=\End(S)$. Then $\End(xR/xJ(R))\cong D$.
\bigskip
\item Let $xR$ be a local module and $D=\End(xR/xJ(R))$ where $x\in R$. Then $\End(S)\cong D$ for every composition factor $S$ of $xR$.
\bigskip
\item Let $xR$, $yR$ be two local modules where $x, y\in R$. If $\End(xR/xJ(R))\not\cong \End(yR/yJ(R))$, then $Hom(xR, yR)=0$.
\end{enumerate}
\end{lemma}

\begin{proof}
\begin{enumerate}[(i)]
\item There exists an $n\ge 1$ such that $xJ(R)^n = 0$ but $xJ(R)^{n-1} \neq 0$. If $n = 1$, then $xR$ is simple, so the result holds. We apply induction on $n$. Suppose $n>1$ and assume that the result holds for $n-1$. Now $xJ(R)J(R)^{n-1} = 0$, but $xJ(R)J(R)^{n-2} \neq 0$. Therefore there exists an element $y \in xJ(R)$ such that $yR$ is local and $yJ(R)^{n-1}= 0$ but $yJ(R)^{n-2}\neq 0$. By the induction hypothesis, $\End(yR/yJ(R))\cong D$. In fact, for any simple submodule $S'$ of $eJ(R)/xJ(R)^2$, $\End(S')\cong D$. Consider the local module $M = xR/xJ(R)^2$. Let $S'$ be a simple submodule of $M$. Then $\Soc(M) = S' \oplus B$ for some $B\subset \Soc(M)$. Then $\End(S')\cong D$. As $A= M/B$ is uniserial, $\Soc(A)\cong S'$ and $A/AJ(R) \cong xR/xJ(R)$. By Lemma \ref{l3}, $\End(xR/xJ(R))\cong D$.  

\bigskip

\item Let $S$ be a simple submodule of $\Soc(xR)$ and $B$ be a complement of $S$ in $xR$. Then $\overline{xR}=xR/B$ is uniform and $\Soc(\overline{xR})\cong S$. By (i), $\End(S)\cong \End(\frac{\overline{xR}}{\overline{xJ(R)}})\cong \End(xR/xJ(R))=D$. Hence $\End(S)\cong D$ for any simple submodule $S$ of $xR$. Let $S_1$ be any composition factor of $xR$. Then there exists a local submodule $yR$ of $xR$ such that $S_1\cong yR/yJ(R)$. By (i), $\End(S_1)\cong \End(S)\cong D$, where $S$ is a simple submodule of $yR$.

\bigskip

\item It is immediate from (ii).  

\end{enumerate}
\end{proof}

Now, we are ready to give the structure of indecomposable right artinian rings of RAI-type.

\begin{theorem} \label{rai}
Let $R$ be an indecomposable right artinian ring of RAI-type. Then the following hold.

\bigskip

\begin{enumerate}[(i)]
\item There exists a division ring $D$ such that $\End(S)\cong D$ for any simple right $R$-module $S$. In particular, $R/J(R)$ is a direct sum of matrix rings over $D$.

\bigskip

\item If $D\not\cong \mathbb Z/2 \mathbb Z$, then every finitely generated indecomposable right $R$-module is quasi-injective. In this case, $R$ is right serial.
\end{enumerate}
\end{theorem}

\begin{proof}
\begin{enumerate}[(i)]
\item Let $e \in R$ be an indecomposable idempotent and $D = eRe/eJ(R)e$. By above lemma, every composition factor $S$ of $eR$ satisfies $\End(S)\cong D$. Now $R_R=\oplus_{i=1}^{n} e_iR$ where $e_i$ are orthogonal indecomposable idempotents with $e_1=e$. Let $A$ be the direct sum of those $e_jR$ for which $\frac{e_jRe_j}{e_jJ(R)e_j}\cong D$. Consider any $e_k$ for which $\frac{e_kRe_k}{e_kJ(R)e_k}\not \cong D$. It follows from Lemma \ref{l4}(iii) that $Ae_kR=0=e_kRA$. Consequently, $A=e_kR$ and we get that $R = A\oplus B$ for some ideal $B$. As $R$ is indecomposable, we get $R = A$. This proves (i). 

\bigskip

\item Suppose $D\not \cong \mathbb Z/2\mathbb Z$. It follows from Corollary \ref{tuk2} that every indecomposable right $R$-module is uniform and quasi-injective. In particular, if $e\in R$ is an indecomposable idempotent, then any homomorphic image of $eR$ is uniform, which gives that $eR$ is uniserial. Hence $R$ is right serial.
\end{enumerate}
\end{proof}

\begin{theorem} \cite{SA} \label{singh}
Let $R$ be a right artinian ring such that $J(R)^2 = 0$. If every finitely generated indecomposable right $R$-module is local, then $R$ satisfies the following conditions.
\begin{enumerate}[(a)]
\item Every uniform right $R$-module is either simple or is injective with composition length 2.
\item $R$ is a left serial ring.
\item For any indecomposable idempotent $e \in R$ either $eJ(R)$ is homogeneous or $l(eJ(R))\le 2$.
\end{enumerate}
Conversely, if $R$ satisfies (a), (b), (c) and $l(eJ(R))\le 2$, then every finitely generated indecomposable right $R$-module is local.
\end{theorem}

\bigskip

\noindent {\bf Example.} Let $R=\left[ 
\begin{array}{ccc}
\mathbb F & \mathbb F & \mathbb F \\ 
0 & \mathbb F & 0 \\ 
0 & 0 & \mathbb F  \\ 
\end{array}
\right] $ where $\mathbb F=\frac{\mathbb Z}{2\mathbb Z}$. \\
Then $R$ is a left serial ring. We have already seen that $e_{11}R$ is an indecomposable module which is automorphism-invariant but not quasi-injective. It follows from Theorem \ref{singh} that every finitely generated indecomposable right $R$-module is local. Thus the only indecomposable modules which are not simple are the homomorphic images of $e_{11}R$, which are $e_{11}R$, $\frac{e_{11}R}{e_{12}\mathbb F}$, and $\frac{e_{11}R}{e_{13}\mathbb F}$. These are all automorphism-invariant. It follows from Theorem \ref{singh} that any finitely generated indecomposable right $R$-module is local. Thus this ring $R$ is an example of a ring where every finitely generated indecomposable right $R$-module is automorphism-invariant.  ~\hfill$\square$    

\bigskip

\noindent {\bf Example.} Let $\mathbb F=\mathbb Z/2 \mathbb Z$ and $R=\left[ 
\begin{array}{cccc}
\mathbb F & \mathbb F & \mathbb F & \mathbb F\\ 
0 & \mathbb F & 0 & 0\\ 
0 & 0 & \mathbb F & 0 \\ 
0 & 0 & 0 & \mathbb F \\ 
\end{array}
\right] $. \\
This ring $R$ is left serial and $J(R)^2=0$. Now $e_{11} J(R)= e_{12}\mathbb F\oplus e_{13}\mathbb F\oplus e_{14}\mathbb F$, a direct sum of non-isomorphic minimal right ideals. It follows from condition (c) in Theorem \ref{singh} that there exists a finitely generated indecomposable right $R$-module that is not local. We have $E_1 = E(e_{12}\mathbb F)$, $E_2 = E(e_{13}\mathbb F)$, $E_3 = E(e_{14}\mathbb F)$, each of them has composition length 2. Now $e_{11}R$ has two homomorphic images $A_1=\frac{e_{11}R}{e_{14}\mathbb F}$ and $A_2=\frac{e_{11}R}{e_{12}\mathbb F}$ such that $\Soc(A_1)\cong e_{12}\mathbb F\oplus e_{13}\mathbb F$ and $\Soc(A_2)\cong e_{13}\mathbb F\oplus e_{14}\mathbb F$. So we get $B_1\subseteq  E_1\oplus E_2\subseteq E_1\oplus E_2\oplus E_3$ such that $A_1\cong B_1$. Similarly, we have $A_2\cong B2\subseteq E_2\oplus E3$. Let $E =E_1\oplus E2\oplus E_3$. Its only automorphism is $I_E$. Thus any essential submodule of $E$ is automorphism-invariant. Now $B = B_1 + B_2 \subset_e E$, so $B$ is automorphism-invariant and $B$ is
not local. We prove that $B$ is indecomposable. We have $B_1\cap B_2 = e_{13} \mathbb F$. Notice that any submodule of $E_1\oplus E_2$ that is indecomposable and not uniserial
is $B_1$. Suppose a simple submodule $S$ of $B$ is a summand of $B$. But $S\subset B_1$ or $S\subset B_2$, therefore $B_1$ or $B_2$ decomposes, which is a contradiction. As $l(B)= 5$, $B$ has a summand $C_1$ with $l(C_1) = 2$. Then $C_1$ being uniserial, it equals one of $E_i$.

Case 1. $C_1 = E_1$. Then $B = C_1\oplus C_2$, where $\Soc(C_2) \cong B_2$. As $C_2$ has no uniserial submodule of length 2, the projection of $B_1$ in $C_2$ equals $\Soc(C_2)$, we get $B_1$ is semi-simple, which is a contradiction.

Similarly other cases follow. Hence $B$ is indecomposable.   ~\hfill$\square$

\bigskip

\noindent Now, we proceed to answer the question of Lee and Zhou \cite{LZ} whether every automorphism-invariant module is pseudo-injective in the affirmative for modules with finite Goldie dimension. 

\begin{theorem} \label{GD}
If $M$ is an automorphism-invariant module with finite Goldie dimension, then $M$ is pseudo-injective. 
\end{theorem}

\begin{proof}
Let $N$ be a submodule of $M$. Let $\sigma: N\rightarrow M$ be a monomorphism. Then $\sigma$ can be extended to a monomorphism $\sigma': E(N)\rightarrow E(M)$. Now, we may write $E(M)=\sigma'(E(N))\oplus P=E(N)\oplus Q$ for some submodules $P$ and $Q$ of $E(M)$. Note that $\sigma'(E(N))\cong E(N)$. Since $M$ has finite Goldie dimension, $E(M)$ has finite Goldie dimension. Thus $E(M)$ is a directly-finite injective module, and hence $E(M)$ satisfies internal cancellation property. Therefore, $P\cong Q$. Thus, there exists an isomorphism $\varphi: Q\rightarrow P$. Now consider the mapping $\lambda: E(M)\rightarrow E(M)$ defined as $\lambda(x+y)=\sigma'(x)+\varphi(y)$ where $x\in E(N)$ and $y\in Q$. Clearly, $\lambda$ is an automorphism of $E(M)$. Since $M$ is assumed to be automorphism-invariant, we have $\lambda(M)\subseteq M$. Thus $\lambda|_{M}:M\rightarrow M$ extends $\sigma$. This shows that $M$ is pseudo-injective.
\end{proof}

\bigskip

\noindent It is well known that if $R$ is a ring such that each cyclic right $R$-module is injective then $R$ is semisimple artinian. For more details on rings characterized by properties of their cyclic modules, the reader is referred to \cite{JST}. We would like to understand the structure of rings for which each cyclic module is automorphism-invariant. In \cite{LZ} it is shown that if every 2-generated right module over a ring $R$ is automorphism-invariant, then $R$ is semisimple artinian.    

\bigskip
 
\noindent A ring $R$ is called a {\it right $\SI$ ring} if every singular right $R$-module is injective \cite{Goodearl}. In \cite{HJL} Huynh, Jain, and L\'{o}pez-Permouth proved that a simple ring $R$ is a right $\SI$ ring if and only if every cyclic singular right $R$-module is $\pi$-injective. Their proof can be adapted to show that a simple right noetherian ring $R$ is a right $\SI$ ring if and only if every cyclic singular right $R$-module is automorphism-invariant. 

\bigskip

\noindent The following lemma due to Huynh et al \cite[Lemma 3.1]{HJL1} is crucial for proving our result. This lemma is, in fact, a generalization of a result of J. T. Stafford given in \cite[Theorem 14.1]{CH}. 

\begin{lemma} $($\cite{HJL1}$)$ \label{stafford}
Let $R$ be a simple right Goldie ring which is not artinian and $M$ a torsion right $R$-module. If $M = aR+bR$, where $bR$ is of finite composition length and $f$ is a non-zero element of $R$ then $M = (a+bxf)R$ for some $x \in R$. 
\end{lemma} 

\bigskip
 
\noindent We know that for a prime right Goldie ring $R$, singular right $R$-modules are the same as torsion right $R$-modules. Now, we are ready to prove the following.

\begin{theorem}
Let $R$ be a simple right noetherian ring. Then $R$ is a right $\SI$ ring if and only if every cyclic singular right $R$-module is automorphism-invariant.
\end{theorem}

\begin{proof}
Let $R$ be a simple right noetherian ring such that every cyclic singular right $R$-module is automorphism-invariant. We wish to show that $R$ is a right $\SI$ ring. If $\Soc(R_R)\neq 0$, then as $R$ is a simple ring, $R=\Soc(R_R)$ and hence $R$ is simple artinian. 

Now, assume $\Soc(R_R)=0$. Let $M$ be any artinian right $R$-module. Since any module is homomorphic image of a free module, we may write $M\cong F/K$ where $F$ is a free right $R$-module. We first claim that $K$ is an essential submodule of $F$. Assume to the contrary that $K$ is not essential in $F$. Let $T$ be a complement of $K$ in $F$. Note that $M \cong \frac{F}{K} \supset \frac{K\oplus T}{K}\cong T$. Since $M$ is an artinian module, $\Soc(M)\neq 0$ and consequently $\Soc(T)\neq 0$. This yields that $\Soc(F)\neq 0$, a contradiction to the assumption that $\Soc(R_R)=0$. Therefore, $K$ is an essential submodule of $F$ and hence $M$ is a singular module. Let $C$ be a cyclic submodule of $M$. We have $\Soc(C)\neq 0$. As $R$ is right noetherian and $C$ is a cyclic right $R$-module, $C$ is noetherian. Thus we have $\Soc(C)=\oplus_{i=1}^{k}S_i$ where each $S_i$ is a simple right $R$-module. By the above lemma, it follows that $C\oplus S_1$ is cyclic. By induction, it may be shown that $C\oplus \Soc(C)$ is cyclic. Now as $C\oplus \Soc(C)$ is a cyclic singular right $R$-module, by assumption $C\oplus \Soc(C)$ is automorphism-invariant. Hence $\Soc(C)$ is $C$-injective. Therefore, $\Soc(C)$ splits in $C$ and hence $C=\Soc(C)\subset M$. Thus $M$ is semisimple. This shows that any artinian right $R$-module $M$ is semisimple.

Now, we prove that every singular module over $R$ is semisimple, or equivalently, for each essential right ideal $E$ of $R$, $R/E$ is semisimple. By the above claim, it suffices to show that $R/E$ is artinian. Set $N=R/E$. If $N$ is not artinian, then we get  $0\subset V_1 \subset N$ with $V_1$ not artinian. Now $N$ is torsion, so is $V_1$. Therefore, $Q=N\oplus V_1$ is torsion and hence cyclic by Lemma \ref{stafford}. Thus we can write $xR=N\oplus V_1$ for some $x\in R$. By the assumption, $xR$ is automorphism-invariant. Hence $V_1$ is $N$-injective. So $N=N_1\oplus V_1$. Repeat the process with $V_1$, so $V_1=N_2\oplus V_2$, where $N_2\neq 0$ and $V_2$ is not artinian. Continuing this process, we get an infinite direct sum of $N_i$ in $N$, which is a contradiction. Thus we conclude that any singular right $R$-module is artinian and consequently semisimple. 

Thus $R$ is a right nonsingular ring such that every singular right $R$-module is semisimple. Hence, by \cite{Goodearl}, $R$ is a right $\SI$ ring.

The converse is obvious.   
\end{proof}

\bigskip

\bigskip

\section{Questions}

\bigskip

\noindent Question 1: Does every automorphism-invariant module satisfy the property C2 ?

\bigskip

\noindent Lee and Zhou \cite{LZ} have shown that every automorphism-invariant module satisfies the property C3.

\bigskip

\noindent Question 2: What is example of an automorphism-invariant module which is not pseudo-injective?

\bigskip

\noindent In Theorem \ref{GD} above, we have shown that such a module cannot have finite Goldie dimension.

\bigskip

\noindent A module $M$ is called skew-injective if for every submodule $N$ of $M$, any endomorphism of $N$ extends to an endomorphism of $M$. In \cite{JST} it is asked whether every skew-injective module with essential socle is quasi-injective. We ask the following

\bigskip

\noindent Question 3: Is every automorphism-invariant module with essential socle a quasi-injective module?  

\bigskip

\noindent Call a ring $R$ to be a {\it right a-ring} if each right ideal of $R$ is automorphism-invariant. 

\bigskip

\noindent Question 4: Describe the structure of a right a-ring.

\bigskip

\noindent Call a ring $R$ to be a {\it right $\Sigma$-a-ring} if each right ideal of $R$ is a finite direct sum of automorphism-invariant right ideals.

\bigskip

\noindent Question 5: Describe the structure of a right $\Sigma$-a-ring.

\bigskip

\noindent Question 6: Let $R$ be a simple ring such that $R_R$ is automorphism-invariant. Must $R$ be a right self-injective ring ?

\bigskip

\noindent In fact, this question is open even when $R_R$ is pseudo-injective (see \cite{ClH}).

\bigskip

\bigskip

\end{document}